\documentclass{amsart}
\usepackage{graphicx}
\usepackage{latexsym,amsmath,amssymb}
\newtheorem{thm}{Theorem}[section]
\newtheorem{cor}[thm]{Corollary}
\newtheorem{lem}[thm]{Lemma}
\newtheorem{exam}[thm]{Example}
\newtheorem{prop}[thm]{Proposition}
\theoremstyle{definition}
\theoremstyle{remark}
\newtheorem{rem}[thm]{Remark}
\numberwithin{equation}{section}
\begin{document}

\title[Weighted conditional type operators ]
{  unbounded weighted conditional type operators on $L^{p}(\Sigma)$ }

\author{\sc\bf Y. Estaremi  }
\address{\sc Y. Estaremi  }
\email{yestaremi@pnu.ac.ir}

\address{Department of Mathematics, Payame Noor University, p. o. box: 19395-3697, Tehran, Iran.}

\thanks{}

\thanks{}



\date{}

\dedicatory{}

\commby{}

%%% ----------------------------------------------------------------------
\begin{abstract}

In this paper  we consider unbounded weighted conditional type operators on the space $L^p(\Sigma)$,  we give some conditions under which they are densely defined and we obtain a dense subset of the domain. Also, we get that a WCT operator is continuous if and only of it is every where defined. A description of polar decomposition, spectrum and spectral radius in this context are provided. Finally, we investigate hyperexpansive WCT operators on the Hilbert space $L^2(\Sigma)$. As a consequence hyperexpansive multiplication operators are investigated.

\end{abstract}

\maketitle

\section{ \sc\bf Introduction }
In this paper we consider a class of unbounded linear operators on $L^p$-spaces having the form $M_wEM_u$, where $E$ is a conditional expectation operator and $M_u$ and $M_w$ are  multiplication operators. What follows is a brief review of the operators $E$ and multiplication operators, along with the notational conventions we will be using.\\
Let $(\Omega,\Sigma,\mu)$ be a $\sigma$-finite measure space and let
$\mathcal{A}$ be a $\sigma$-subalgebra of $\Sigma$ such that
$(\Omega,\mathcal{A},\mu)$ is also $\sigma$-finite. We denote the
collection of (equivalence classes modulo sets of zero measure of)
$\Sigma$-measurable complex-valued functions on $\Omega$ by
$L^0(\Sigma)$ and the support of a function $f\in L^0(\Sigma)$ is
defined as $S(f)=\{t\in \Omega; f(t)\neq 0\}$. Moreover, we set
$L^p(\Sigma)=L^p(\Omega,\Sigma,\mu)$. We also
adopt the convention that all comparisons between two functions or
two sets are to be interpreted as holding up to a $\mu$-null set.
For each $\sigma$-finite subalgebra $\mathcal{A}$ of $\Sigma$, the
conditional expectation, $E^{\mathcal{A}}(f)$, of $f$ with respect
to $\mathcal{A}$ is defined whenever $f\geq0$ or
$f\in L^p(\Sigma)$. In any case, $E^{\mathcal{A}}(f)$
is the unique $\mathcal{A}$-measurable function for which
%\vspace*{0.3cm} For a sub-$\sigma$-finite algebra
%$\mathcal{A}\subseteq\Sigma$, the conditional expectation
%operator associated with $\mathcal{A}$ is the mapping
%$f\rightarrow E^{\mathcal{A}}f$, defined for all non-negative $f$
%as well as for all $f\in L^p(\Sigma)$, $1\leq p\leq \infty$,
%where $E^{\mathcal{A}}f$, by the Radon-Nikodym theorem, is the
%unique $\mathcal{A}$-measurable function satisfying
$$\int_{A}fd\mu=\int_{A}E^{\mathcal{A}}fd\mu, \ \ \ \forall A\in \mathcal{A} .$$
 As an operator on
$L^{p}({\Sigma})$, $E^{\mathcal{A}}$ is an idempotent and
$E^{\mathcal{A}}(L^p(\Sigma))=L^p(\mathcal{A})$. If there is no
possibility of confusion we write $E(f)$ in place of
$E^{\mathcal{A}}(f)$ \cite{rao,z}. This operator will play a major role in our
work and we list here some of its useful properties:

\vspace*{0.2cm} \noindent $\bullet$ \  If $g$ is
$\mathcal{A}$-measurable, then $E(fg)=E(f)g$.

\noindent $\bullet$ \ $|E(f)|^p\leq E(|f|^p)$.

\noindent $\bullet$ \ If $f\geq 0$, then $E(f)\geq 0$; if $f>0$,
then $E(f)>0$.

\noindent $\bullet$ \ $|E(fg)|\leq
E(|f|^p)|^{\frac{1}{p}}E(|g|^{q})|^{\frac{1}{q}}$, (H\"{o}lder inequality) for all $f\in L^p(\Sigma)$ and $g\in L^q(\Sigma)$, in which $\frac{1}{p}+\frac{1}{q}=1$.

\noindent $\bullet$ \ For each $f\geq 0$, $S(f)\subseteq S(E(f))$.\\
 Let $u\in L^0(\Sigma)$. The corresponding multiplication operator $M_u$ on $L^p(\Sigma)$ is defined by $f\rightarrow uf$. Our interest in operators of the form $M_wEM_u$ stems from the fact that such products tend to appear often in the study of those operators related to conditional expectation. This observation was made in \cite{dhd,dou,gd,lam,mo}. In this paper, first we investigate some properties of unbounded weighted conditional type operators on the space $L^p(\Sigma)$ and then we study hyperexpansive ones.

\section{ \sc\bf Unbounded weighted conditional type operators}
Let $X$ stand for a Banach space and $\mathcal{B}(X)$ for the Banach algebra of all linear operators
on $X$. By an operator in $X$ we understand a
linear mapping $T:\mathcal{D}(T)\subseteq X\rightarrow
X$ defined on a linear subspace $\mathcal{D}(T)$ of
$X$ which is called the domain of $T$. The linear map $T$ is called densely defined if $\mathcal{D}(T)$ is dense in
$X$ and it is called closed if its graph ($\mathcal{G}(T)$) is closed in $X\times X$, where $\mathcal{G}(T)=\{(f,Tf): f\in \mathcal{D}(T)\}$.
We studied bounded weighted conditional type operators on $L^p$-spaces in \cite{ej}. Also we investigated unbounded weighted conditional type operators of the form $EM_u$ on the Hilbert space $L^2(\Sigma)$ in \cite{es}. Here we investigate unbounded  weighted conditional type operators of the form $M_wEM_u$ on $L^p$-spaces. Let $f$ be a positive $\Sigma$-measurable function on $\Omega$. Define
the measure $\mu_f:\Sigma\rightarrow [0, \infty]$  by
$$\mu_f(E)=\int_{E}fd\mu, \ \ \ \ E\in \Sigma.$$
It is clear that the measure $\mu_f$ is also $\sigma$-finite, since $\mu$  is $\sigma$-finite. From now on we assume that $u$ and $w$ are conditionable (i.e., $E(u)$ and $E(w)$ are defined). Operators of the form $M_wEM_u(f)=wE(u.f)$ acting in
$L^p(\mu)$ with $\mathcal{D}(M_wEM_u)=\{f\in
L^p(\mu):wE(u.f) \in L^p(\mu)\}$ are called weighted conditional type operators (or briefly WCT operators). In the first proposition we give a condition under which the WCT operator $M_wEM_u$ is densely defined.
\begin{thm}\label{t1}
For $1\leq p,q<\infty$ such that $\frac{1}{p}+\frac{1}{q}=1$  and $E(|w|^p)^{\frac{1}{p}}E(|u|^q)^{\frac{1}{q}}<\infty$ a.e., the linear
transformation $M_wEM_u$ is densely defined.
\end{thm}
\begin{proof}
For each $n\in \mathbb{N}$, define
$$A_n=\{t\in \Omega:n-1\leq E(|w|^p)(t)E(|u|^q)^{\frac{p}{q}}(t)<n\}.$$
It is clear that each $A_n$ is $\mathcal{A}$-measurable and $\Omega$ is
expressible as the disjoint union of sets in the sequence
$\{A_n\}^{\infty}_{n=1}$, $\Omega=\cup^{\infty}_{n=1}A_n$.\\
Let $f\in L^p(\Sigma)$ and $\epsilon>0$. Then, there exists $N>0$
such that
$$\int_{\cup^{\infty}_{n=N}A_n}|f|^pd\mu=\sum^{\infty}_{n=N}\int_{A_n}|f|^pd\mu<\epsilon.$$
Define the sets
$$B_N=\cup^{\infty}_{n=N}A_n, \ \ \ \ \ \
C_N=\cup^{N-1}_{n=1}A_n.$$ Then, $\int_{B_N}|f|^pd\mu<\epsilon$
and $C_N=\{t\in \Omega:E(|w|^p)(t)E(|u|^q)^{\frac{p}{q}}(t)<N-1\}$.
Next, we define $g=f.\chi_{C_N}$. Clearly $g\in L^p(\Sigma)$ and
$E(g)=E(f).\chi_{C_N}$. Now, we show that $g\in \mathcal{D}=\mathcal{D}(M_wEM_u)$.
Consider the following:
\begin{align*}
\int_{\Omega}|wE(ug)|^pd\mu&=\int_{\Omega}|wE(uf)\chi_{C_N}|^pd\mu\\
&=\int_{C_N}E(|w|^p)|E(uf)|^pd\mu\\
&\leq\int_{C_N}E(|w|^p)E(|u|^q)^{\frac{p}{q}}|f|^pd\mu\\
&\leq (N-1)\int_{C_N}|f|^pd\mu<\infty.
\end{align*}
Thus, $wE(uf)\in L^p(\Sigma)$. Now, we show that $\|g-f\|_p<\epsilon$:
\begin{align*}
\|g-f\|^p_{p}&=\int_{X}|g-f|^pd\mu\\
&=\int_{C_N}|g-f|^pd\mu<\epsilon.
\end{align*}
Thus, $\mathcal{D}$ is dense in $L^p(\Sigma)$.
\end{proof}
Here we obtain a dense subset of $L^p(\mu)$ that we need it to proof our next results.
\begin{lem}\label{l1}
Let $1<p,q<\infty$ such that $\frac{1}{p}+\frac{1}{q}=1$, $J=1+E(|w|^p)E(|u|^q)^{\frac{p}{q}}$, $E(|w|^p)^{\frac{1}{p}}E(|u|^q)^{\frac{1}{q}}<\infty$ a.e, $\mu$, and $d\nu=J d\mu$. We get that $S(J)=\Omega$ and\\
(i) $L^p(\nu)\subseteq\mathcal{D}(M_wEM_u)$,\\
(ii) $\overline{L^p(\nu)}^{\|.\|_{\mu}}=\overline{\mathcal{D}(M_wEM_u)}^{\|.\|_{\mu}}=L^p(\mu)$.
\end{lem}
\begin{proof}
 Let $f\in L^p(\nu)$. Then
 $$
 \|f\|^pd\mu\leq \|f\|^p_{\nu}<\infty,$$
so $f\in L^p(\mu)$. Also, by conditional-type H\"{o}lder-inequality we have
\begin{align*}
\|M_wEM_u(f)\|^pd\mu&\leq\int_{\Omega}E(|w|^p)E(|u|^q)^{\frac{p}{q}}E(|f|^p)d\mu\\
&=\int_{\Omega}E(|w|^p)E(|u|^q)^{\frac{p}{q}}|f|^pd\mu\\
&\leq\|f\|^p_{\nu}<\infty,
\end{align*}
this implies that $f\in\mathcal{D}(M_wEM_u)$. Now we prove that $L^p(\nu)$ is dense in $L^p(\mu)$. By Riesz representation theorem we have
$$(L^p(\nu))^{\perp}=\{g\in L^q(\mu): \int_{\Omega} f.gd\mu=0, \ \ \ \forall f\in L^p(\nu)\}.$$
Suppose that $g\in(L^p(\nu))^{\perp}$. For $A\in \Sigma$ we set $A_n=\{t\in A:J(t)\leq n\}$. It is clear that $A_n\subseteq A_{n+1}$ and $\Omega=\cup^{\infty}_{n=1}A_n$. Also, $\Omega$ is $\sigma$-finite, hence $\Omega=\cup^{\infty}_{n=1}\Omega_n$ with $\mu(\Omega_n)<\infty$. If we set $B_n=A_n\cap \Omega_n$, then $B_n\nearrow A$ and so $g.\chi_{B_n}\nearrow g.\chi_{A}$ a.e. $\mu$. Since $\nu(B_n)\leq (n+1) \mu(B_n)<\infty$, we have $\chi_{B_n}\in L^p(\nu)$ and by our assumption $\int_{B_n}fd\mu=0$. Therefore by Fatou's lemma we get that $\int_{A}gd\mu=0$. Thus for all $A\in \Sigma$ we have $\int_{A}gd\mu=0$. This means that $g=0$ a.e. $\mu$ and so $L^p(\nu)$ is dense in $L^p(\mu)$.
\end{proof}
By the Lemma \ref{l1} we get that $L^p(\nu)$ is a core of $M_wEM_u$.
Here we give a condition that we will use in the next theorem.\\
$(\bigstar)$ If $(\Omega,\mathcal{A},\mu)$ is a $\sigma$-finite measure space and $J-1=(E(|u|^q))^{\frac{p}{q}}E(|w|^p)<\infty$ a.e. $\mu$, then there exists a sequence $\{A_n\}^{\infty}_{n=1}\subseteq \mathcal{A}$ such that $\mu(A_n)<\infty$ and $J-1<n$ a.e. $\mu$ on $A_n$ for every $n\in \mathbb{N}$ and $A_n\nearrow \Omega$ as $n\rightarrow \infty$.

\begin{thm} \label{t2}If $u,w:\Omega\rightarrow \mathbb{C}$ are $\Sigma$-measurable and  $1<p,q<\infty$ such that $\frac{1}{p}+\frac{1}{q}=1$, then the following conditions are equivalent:\\

(i) $M_wEM_u$ is densely defined on $L^p(\Sigma)$,\\

(ii) $J-1=E(|w|^p)(E(|u|^q))^{\frac{p}{q}}<\infty$ a.e., $\mu$.\\

(iii) $\mu_{J-1}\mid_{\mathcal{A}}$ is $\sigma$-finite.
\end{thm}
\begin{proof}
  $(i)\rightarrow (ii)$ Set $E=\{E(|w|^p)(E(|u|^q))^{\frac{p}{q}}=\infty\}$. Clearly by Lemma  (i), $f\mid_{E}=0$ a.e., $\mu$ for every $f\in L^p(\nu)$. This implies that  $f.J\mid_{E}=0$ a.e., $\mu$ for every $f\in L^p(\mu)$. So we have $J.\chi_{A\cap E}=0$ a.e., $\mu$ for all $A\in \Sigma$ with $\mu(A)<\infty$. By the $\sigma$-finiteness of $\mu$ we have $J.\chi_{E}=0$ a.e., $\mu$. Since $S(J)=\Omega$, we get that $\mu(E)=0$.\\
$(ii)\rightarrow (i)$ Evident.\\

$(ii)\rightarrow (iii)$ Let $\{A_n\}^{\infty}_{n=1}$ be in $(\bigstar)$. We have

$$\mu_{J-1}\mid_{\mathcal{A}}(A_n)=\int_{A_n}E(|w|^p)(E(|u|^q))^{\frac{p}{q}}d\mu\leq n\mu(A_n)<\infty, \ \ \ \ n\in \mathbb{N}.$$
This yields (iii).\\

$(iii)\rightarrow (i)$ Let $\{A_n\}^{\infty}_{n=1}\subseteq \mathcal{A}$ be a sequence such that $A_n\nearrow \Omega$ as $n\rightarrow \infty$ and $\mu_{J-1}\mid_{\mathcal{A}}(A_n)<\infty$ for every $k\in \mathbb{N}$. It follows from the definition of $\mu_{J-1}$ that $J-1=E(|w|^p)(E(|u|^q))^{\frac{p}{q}}<\infty$ a.e., $\mu$ on $\Omega$. Applying Theorem \ref{t1}, we obtain $(i)$.
\end{proof}
Let $X, Y$ be Banach spaces and $T:X\rightarrow Y$ be a linear operator. If $T$ is densely defined, then there is a unique maximal operator $T^{\ast}$ from $\mathcal{D}(T^{\ast})\subset Y^{\ast}$ into $X^{\ast}$ such that
$$y^{\ast}(Tx)=\langle Tx, y^{\ast}\rangle=\langle x, T^{\ast}y^{\ast}\rangle=T^{\ast}y^{\ast}(x), \ \ \  x\in \mathcal{D}(T), \ \ \ y^{\ast}\in\mathcal{D}(T^{\ast}).$$
$T^{\ast}$ is called the adjoint of $T$.\\
 By Riesz representation theorem for $L^p$- spaces we have $\langle f, F\rangle=F(f)=\int_\Omega f \bar{F}d\mu$, when $f\in L^p(\Sigma)$, $F\in L^q(\Sigma)=(L^p(\Sigma))^{\ast}$ and $\frac{1}{p}+\frac{1}{q}=1$. By the Theorem \ref{t2} easily we get that: the operator $M_wEM_u$ is densely defined if and only if the operator $M_{\bar{u}}EM_{\bar{w}}$ is densely defined. In the next proposition we obtain the adjoint of the WCT  operator $M_wEM_u$ on the Banach space $L^p(\Sigma)$.
\begin{prop}\label{p1}
If the linear
transformation $T=M_wEM_u$ is densely defined on $L^p(\Sigma)$, then $M_{\bar{u}}EM_{\bar{w}}$ is a densely defined operators on $L^q(\Sigma)$  and  $T^{\ast}=M_{\bar{u}}EM_{\bar{w}}$, where $\frac{1}{p}+\frac{1}{q}=1$.
\end{prop}
\begin{proof}
Let $f\in \mathcal{D}(T)$  and $g\in \mathcal{D}(T^{\ast})$. So we have
\begin{align*}
\langle Tf,g\rangle&=\int_{\Omega}wE(uf)\bar{g}d\mu\\
&=\int_{\Omega}fuE(w\bar{g})d\mu\\
&=\langle f,M_{\bar{u}}EM_{\bar{w}}g\rangle.
\end{align*}
Hence $T^{\ast}=M_{\bar{u}}EM_{\bar{w}}$.
\end{proof}
Now we prove that every densely defined WCT operator is closed.
\begin{prop} \label{p2} If $(E(|u|^q))^{\frac{p}{q}}E(|w|^p)<\infty$ a.e., $\mu$. Then the linear transformation $M_wEM_u:\mathcal{D}(M_wEM_u)\rightarrow L^p(\Sigma)$ is closed.
\end{prop}
\begin{proof}
  Assume that $f_n\in \mathcal{D}(M_{w}EM_{u})$, $f_n\rightarrow f$, $wE(uf_n)\rightarrow g$, and let $h\in  \mathcal{D}(M_{\bar{u}}EM_{\bar{w}})$. Then

$$\langle f,M_{\bar{u}}EM_{\bar{w}}h\rangle=\lim_{n\rightarrow \infty}\langle f_n,M_{\bar{u}}EM_{\bar{w}}h\rangle$$

$$=\lim_{n\rightarrow \infty}\langle wE(uf_n),h\rangle=\langle g,h\rangle.$$

This calculation ( which uses the continuity of the inner product and the fact that $f_n\in \mathcal{D}(M_{w}EM_{u})$) shows that $f\in \mathcal{D}(M_{w}EM_{u})$ and $wE(uf)=g$, as required.
\end{proof}
In the next theorem we get that if WCT operator $M_wEM_u$ is densely defined, then it is continuous if and only it is every where defined.

\begin{thm}\label{t3}
If $(E(|u|^q))^{\frac{p}{q}}E(|w|^p)<\infty$ a.e., $\mu$. Then the WCT operator  $M_wEM_u:\mathcal{D}(M_wEM_u)\rightarrow L^p(\Sigma)$ is continuous if and only if it is every where defined i.e., $\mathcal{D}(M_wEM_u)=L^p(\Sigma)$.
\end{thm}
\begin{proof} Let $M_wEM_u$ be continuous. By Lemma \ref{l1} it is closed. Hence easily we get that $\mathcal{D}(M_wEM_u)$ is closed and so $\mathcal{D}(M_wEM_u)=L^p(\Sigma)$. The converse is easy by closed graph theorem.
\end{proof}
We denote the range of the operator $T$ as $\mathcal{R}(T)$ i.e., $\mathcal{R}(T)=\{T(x): x\in \mathcal{D}(T)\}$.

\begin{prop} If $E(|u|^2)E(|w|^2)<\infty$ a.e., $\mu$ and $M_wEM_u:\mathcal{D}(M_wEM_u)\subset L^2(\Sigma)\rightarrow L^2(\Sigma)$, then $\mathcal{R}(M_wEM_u)$ is closed if and only if $\mathcal{R}(M_{\bar{u}}EM_{\bar{w}})$ is closed.
\end{prop}
\begin{proof}
Let $P:L^2(\Sigma)\times L^2(\Sigma)\rightarrow \mathcal{G}(M_wEM_u)$ be a projection and $Q:L^2(\Sigma)\times L^2(\Sigma)\rightarrow \{0\}\times L^2(\Sigma)$ be the canonical projection. It is clear that $\mathcal{R}(M_wEM_u)\cong\mathcal{R}(QP)$. Also, $\mathcal{R}(M_{\bar{u}}EM_{\bar{w}})\cong\mathcal{R}((I-Q)(I-P))$. Since $P$ and $Q$ are orthogonal projections, then $\mathcal{R}(QP)$ is closed if and only if $\mathcal{R}((I-Q)(I-P))$. Thus we obtain the desired result.
\end{proof}
It is well-known that for a densely defined closed operator $T$ of $\mathcal{H}_1$ into $\mathcal{H}_2$, there exists a partial isometry $U_T$ with initial space $\mathcal{N}(T)^{\perp}=\overline{\mathcal{R}(T^{\ast})}=\overline{\mathcal{R}(|T|)}$
and final space $\mathcal{N}(T^{\ast})^{\perp}=\overline{\mathcal{R}(T)}$ such that $$T=U_T|T|.$$
\begin{thm}
Suppose that  $\mathcal{D}(M_wEM_u)$
is dense in $L^2(\Sigma)$. Let $M_wEM_u=U|M_wEM_u|$ be the polar decomposition of $M_wEM_u$. Then\\

(i) $|M_wEM_u|=M_{u'}EM_u$, where $u'=(\frac{E(|w|^2)}{E(|u|^2)})^{\frac{1}{2}}.\chi_{S}.\bar{u}$ and $S=S(E(|u|^2))$,\\

(ii) $U=M_{w'}EM_{u}$, where $w':\Omega\rightarrow \mathbb{C}$ is an a.e. $\mu$ well-defined $\Sigma$-measurable function such that
$$w'=\frac{w}{(E(|w|^2)E(|u|^2))^{\frac{1}{2}}}.\chi_{S\cap G},$$
in which $G=S(E(|w|^2))$.
\end{thm}
\begin{proof}
 (i). For every $f\in \mathcal{D}(M_{u'}EM_u)$ we have
\begin{align*}
\|M_{u'}EM_u(f)(f)\|^2=\||M_wEM_u|(f)\|^2.
\end{align*}
Also, by Lemma \ref{l1} we conclude that $\mathcal{D}(M_{u'}EM_u)=\mathcal{D}(|M_wEM_u|)$ and it is easily seen that $M_{u'}EM_u$ is a positive operator. These observations imply that $|M_wEM_u|=M_{u'}EM_u$.\\
(ii). For $f\in L^2(\Sigma)$ we have
$$\int_{\Omega}|w'E(uf)|^2d\mu=\int_{\Omega}\frac{\chi_{S\cap G}}{E(|w|^2)E(|u|^2)}|wE(uf)|^2d\mu,$$
which implies that the operator $M_{w'}EM_{u}$ is well-defined and $\mathcal{N}(M_wEM_u)=\mathcal{N}(M_{w'}EM_{u})$. Also, for $f\in \mathcal{D}(M_wEM_u)\ominus \mathcal{N}(M_wEM_u)$ we have
\begin{align*}
U(|M_wEM_u|(f))=wE(uf).\chi_{S\cap G}=wE(uf).
\end{align*}
Thus $\|U(f)\|=\|f\|$ for all $f\in \mathcal{R}(|M_wEM_u|)$ and since $U$ is a contraction, then it holds for all $f\in \mathcal{N}(M_wEM_u)^{\perp}=\overline{\mathcal{R}(|M_wEM_u|)}$.
\end{proof}
Here we remind that: if $T:\mathcal{D}(T)\subset X\rightarrow X$ is a closed linear operator on the Banach space $X$, then a complex number $\lambda$ belongs  to the resolvent set $\rho(T)$ of $T$, if the operator $\lambda I-T$ has a bounded everywhere on $X$ defined inverse $(\lambda I-T)^{-1}$, called the resolvent of $T$ at $\lambda$ and denoted by $R_{\lambda}(T)$. The set $\sigma(T):=\mathbb{C}\setminus \rho(T)$ is called the spectrum of the operator $T$.\\
It is known that, if $a,b$ are elements of a unital algebra $A$, then $1-ab$ is invertible if and only if $1-ba$ is invertible. A consequence of this equivalence is that $\sigma(ab)\setminus \{0\}=\sigma(ba)\setminus \{0\}$.  Now, in the next theorem we compute the spectrum of WCT operator $M_wEM_u$ as a densely defined operator on $L^2(\Sigma)$.

\begin{prop}\label{p4} Let $M_wEM_u$ be densely defined and $\mathcal{A}\varsubsetneq\Sigma$, then\\

(i) $\operatorname{ess range(E(uw))}\setminus \{0\}\subseteq \sigma(M_wEM_u)$,\\

(ii) If $L^2(\mathcal{A})\subseteq \mathcal{D}(EM_{uw})$, then $\sigma(M_wEM_u)\setminus \{0\}\subseteq \operatorname{ess range(E(uw))}\setminus \{0\}$.
\end{prop}
\begin{proof} Since $\sigma(M_wEM_u)\setminus \{0\}=\sigma(EM_{uw})\setminus \{0\}$, then by using theorem 2.8 of \cite{es} we get the proof.
\end{proof}
By a similar method that we used in the proof of theorem 2.8 of \cite{es} we have the same assertion for the spectrum of the densely defined operator $EM_u$ on the space $L^p(\Sigma)$, i.e.,

(i) $\operatorname{ess range(E(u))}\cup\{0\}\subseteq \sigma(EM_u)$,\\

(ii) If $L^p(\mathcal{A})\subseteq \mathcal{D}(EM_u)$, then $\sigma(EM_u)\subseteq \operatorname{ess range(E(u))}\cup\{0\}$.\\
By these observations we have the next remark.
\begin{rem} Let $M_wEM_u$ be densely defined operator on $L^p(\Sigma)$ and $\mathcal{A}\varsubsetneq\Sigma$, then\\

(i) $\operatorname{ess range(E(uw))}\setminus \{0\}\subseteq \sigma(M_wEM_u)$,\\

(ii) If $L^p(\mathcal{A})\subseteq \mathcal{D}(EM_{uw})$, then $\sigma(M_wEM_u)\setminus \{0\}\subseteq \operatorname{ess range(E(uw))}\setminus \{0\}$.
\end{rem}
As we know the spectral radius of a densely defined operator $T$ is denoted by $r(T)$ and is defined as: $r(T)=\sup_{\lambda\in \sigma(T)} |\lambda|$. Hence we have the next corollary.
\begin{cor}
If the WCT operator $M_wEM_u$ is densely defined on $L^p(\Sigma)$ and $L^p(\mathcal{A})\subseteq \mathcal{D}(EM_{uw})$, then
$\sigma(M_wEM_u)\setminus \{0\}= \operatorname{ess range(E(uw))}\setminus \{0\}$ and $r(M_wEM_u)=\|E(uw)\|_{\infty}$.
\end{cor}
A densely defined operator $T$  on the Hilbert space $\mathcal{H}$ is said to be {\it hyponormal} if $\mathcal{D}(T)\subseteq \mathcal{D}(T^{\ast})$ and $\|T^{\ast}(f)\|\leq \|T(f)\|$ for $f\in \mathcal{D}(T)$. A densely defined operator $T$  on the Hilbert space $\mathcal{H}$ is said to be {\it normal} if T is closed and $T^{\ast}T=TT^{\ast}$.
 For the WCT operator $T=M_wEM_u$ on $L^2(\Sigma)$ we have $T^{\ast}=M_{\bar{u}}EM_{\bar{w}}$ and we recall that $T$ is densely defined if and only if $T^{\ast}$ is densely defined. If $T$ is densely defined, then by the Lemma \ref{l1} we get that
 $L^2(\nu)\subseteq \mathcal{D}(T)$, $L^2(\nu)\subseteq \mathcal{D}(T^{\ast})$  and
$$\overline{L^2(\nu)}^{\|.\|_{\mu}}=\overline{\mathcal{D}(T)}^{\|.\|_{\mu}}=\overline{\mathcal{D}(T^{\ast})}^{\|.\|_{\mu}}=L^2(\mu),$$
in which $d\nu=Jd\mu$ and $J=1+E(|w|^2)E(|u|^2)$.
 Also, we have
 $T^{\ast}T=M_{E(|w|^2)\bar{u}}EM_u$ and $TT^{\ast}=M_{E(|u|^2)w}EM_{\bar{w}}$. Similarly, we have
$L^2(\nu')\subseteq \mathcal{D}(T^{\ast}T)$, $L^2(\nu')\subseteq \mathcal{D}(TT^{\ast})$  and
$$\overline{L^2(\nu')}^{\|.\|_{\mu}}=\overline{\mathcal{D}(T^{\ast}T)}^{\|.\|_{\mu}}=\overline{\mathcal{D}(TT^{\ast})}^{\|.\|_{\mu}}=L^2(\mu),$$
in which $d\nu'=J'd\mu$ and $J'=1+(E(|w|^2))^2(E(|u|^2))^2$. By these observations we have next assertions.
\begin{prop}\label{p27}
Let WCT operator $M_wEM_u$ be densely defined on $L^2(\Sigma)$. Then we have the followings:

(i) If $u(E(|w|^2))^{\frac{1}{2}}=\bar{w}(E(|u|^2))^{\frac{1}{2}}$ with respect to the measure $\mu$, then $T=M_wEM_u$ is normal.\\

(ii) If $T=M_wEM_u$ is normal, then $E(|w|^2)|E(u)|^2=E(|u|^2)|E(w)|^2$ with respect to the measure $\mu$.
\end{prop}
\begin{proof}
(i) Direct computations shows that
$$T^{\ast}T-TT^{\ast}=M_{\bar{u}E(|w|^{2})}EM_{u}-M_{wE(|u|^{2})}EM_{\bar{w}},$$
on $L^{2}(\nu')$.
Hence for every $f\in L^{2}(\nu')$ we have
\begin{align*}
\langle T^{\ast}T-TT^{\ast}(f),f\rangle&=\int_{X}E(|w|^2)E(uf)\bar{uf}-E(|u|^2)E(\bar{w}f)w\bar{f}d\mu\\
&=\int_{X}|E(u(E(|w|^2))^{\frac{1}{2}}f)|^2-|E((E(|u|^2))^{\frac{1}{2}}\bar{w}f)|^2d\mu.
\end{align*}
This implies that if
$$(E(|u|^2))^{\frac{1}{2}}\bar{w}=u(E(|w|^2))^{\frac{1}{2}},$$
then for all $f\in L^2(\nu')$, $\langle
T^{\ast}T-TT^{\ast}(f),f\rangle=0$. Thus $T^{\ast}T=TT^{\ast}$.\\

(ii) Suppose that $T$ is normal. By (i), for all $f\in L^2(\nu')$
we have
$$\int_{X}|E(u(E(|w|^2))^{\frac{1}{2}}f)|^2-|E((E(|u|^2))^{\frac{1}{2}}\bar{w}f)|^2d\mu=0.$$
Let $A\in \mathcal{A}$, with $0<\nu'(A)<\infty$. By replacing $f$
to $\chi_{A}$, we have

$$\int_{A}|E(u(E(|w|^2))^{\frac{1}{2}})|^2-|E((E(|u|^2))^{\frac{1}{2}}\bar{w})|^2d\mu=0$$
and so
$$\int_{A}|E(u)|^2E(|w|^2)-|E(w)|^2E(|u|^2)d\mu=0.$$
Since $A\in \mathcal{A}$ is arbitrary and $\mu\ll \nu'$ (absolutely continuous), then
$|E(u)|^2E(|w|^2)=|E(w)|^2E(|u|^2)$ with respect to $\mu$.
\end{proof}

\begin{prop}
Let the WCT operator $M_wEM_u$ be densely defined on $L^2(\Sigma)$. Then we have the followings:\\

(i) If $u(E(|w|^2))^{\frac{1}{2}}\geq\bar{w}(E(|u|^2))^{\frac{1}{2}}$ with respect to $\mu$, then $T=M_wEM_u$ is hyponormal.\\

(ii) If $T=M_wEM_u$ is hyponormal, then $E(|w|^2)|E(u)|^2\geq E(|u|^2)|E(w)|^2$ with respect to the measure $\mu$.
\end{prop}
\begin{proof}
By a similar method of \ref{p27} we can get the proof.
\end{proof}
\section{ \sc\bf Hyperexpansive WCT operators}

In this section we are going to present conditions under which WCT operator $M_wEM_u$ on $L^2(\Sigma)$ is $k$-isometry, $k$-expansive, $k$-hyperexpansive and completely hyperexpansive.  For an operator $T$ on the Hilbert space $\mathcal{H}$ we set

$$\Theta_{T,n}(f)=\sum_{0\leq i\leq n}(-1)^i\left(%
\begin{array}{c}
  n \\
  i \\
\end{array}%
\right)\|T^i(f)\|^2, \ \ \ \ \ f\in \mathcal{D}(T^n), \ \
n\geq1.$$

By means of this definition  an operator $T$ on $\mathcal{H}$ is said to be:\\

(i) $k$-isometry $(k\geq1)$ if $\Theta_{T,k}(f)=0$ for $f\in
\mathcal{D}(T^k)$,\\

(ii) $k$-expansive $(k\geq1)$ if $\Theta_{T,k}(f)\leq0$ for $f\in
\mathcal{D}(T^k)$,\\

(iii) $k$-hyperexpansive $(k\geq1)$ if $\Theta_{T,n}(f)\leq0$ for
$f\in
\mathcal{D}(T^n)$ and $n=1,2,...,k$.\\

(iv) completely hyperexpansive  if $\Theta_{T,n}(f)\leq0$ for
$f\in
\mathcal{D}(T^n)$ and $n\geq1$.\\
For more details one can see \cite{jab,jjst,sho}. It is easily seen that for every $f\in L^2(\Sigma)$
$$\|M_wEM_u(f)\|_2=\|EM_v(f)\|_2,$$
where $v=u(E(|w|^2))^{\frac{1}{2}}$. Thus without loss of generality we can consider the operator $EM_v$
instead of $M_wEM_u$ in our discussion. First we recall some concepts that we need them in the sequel.  Now we present our main results. The next lemma is a direct consequence of Theorem \ref{t2}.\\
\begin{lem}\label{l31}
     For every $n\in \mathbb{N}$ the operator $(EM_v)^n$ on $L^2(\Sigma)$ is densely-defined if and only if the operator $EM_v$ is densely defined on $L^2(\Sigma)$.
\end{lem}
In the Theorem \ref{t31} we give some necessary and sufficient conditions for $k$-isometry and $k$-expansive WCT operators $EM_v$.

\begin{thm} \label{t31}If $\mathcal{D}(EM_v)$
is dense in $L^2(\mu)$, then:\\

(i) If the operator $EM_v$ is $k$-isometry $(k\geq1)$, then $A^0_{k}(|E(v)|^2)=0$;\\

(ii) If $(1+E(|v|^2)A^1_{k}(|E(v)|^2))=0$ and $|E(vf)|^2=E(|v|^2)E(|f|^2)$ for all $f\in \mathcal{D}(EM_v)$, then the operator $EM_v$ is $k$-isometry;\\

(iii) If the operator $EM_v$ is $k$-expansive, then $A^0_{k}(|E(v)|^2)\leq0$;\\

(iv) If $(1+E(|v|^2)A^1_{k}(|E(v)|^2))\leq0$ and $|E(vf)|^2=E(|v|^2)E(|f|^2)$ for all $f\in \mathcal{D}(EM_v)$, then the operator $EM_v$ is $k$-expansive, where
$$A^0_{k}(|E(v)|^2)=\sum_{0\leq i\leq k}(-1)^i
\left(%
\begin{array}{c}
  k \\
  i \\
\end{array}
\right)|E(v)|^{2i}, \ \ \
A^1_{k}(|E(v)|^2)=\sum_{1\leq i\leq k}(-1)^i
\left(%
\begin{array}{c}
  k \\
  i \\
\end{array}%
\right)|E(v)|^{2(i-1)}.$$
\end{thm}
\begin{proof}
 Suppose that the operator $EM_v$ is $k$-isometry. So for all $f\in \mathcal{D}((EM_v)^k)$ we have
\begin{align*}
0&=\Theta_{T,k}(f)\\
&=\sum_{0\leq i\leq k}(-1)^i
\left(%
\begin{array}{c}
  n \\
  i \\
\end{array}%
\right)\|(EM_v)^i(f)\|^2\\
&=\int_{\Omega}|f|^2d\mu+\sum_{1\leq i\leq k}(-1)^i
\left(%
\begin{array}{c}
  n \\
  i \\
\end{array}%
\right)\int_{\Omega}|E(v)|^{2(i-1)}|E(vf)|^2d\mu,\\
\end{align*}
and so for all $\mathcal{A}$-measurable functions $f\in \mathcal{D}((EM_v)^k)$
\begin{align*}
0&=\int_{\Omega}|f|^2d\mu+\sum_{1\leq i\leq k}(-1)^i
\left(%
\begin{array}{c}
  n \\
  i \\
\end{array}%
\right)\int_{\Omega}|E(v)|^{2(i-1)}|E(v)|^2|f|^2d\mu\\
&=\int_{\Omega}\left(\sum_{0\leq i\leq k}(-1)^i
\left(%
\begin{array}{c}
  n \\
  i \\
\end{array}%
\right)|E(v)|^{2i}\right)|f|^2d\mu.\\
\end{align*}
Since $(EM_v)^k$ is densely defined, then we get that $A_{k}(|E(v)|^2)=0$.\\
(ii) Let $1+E(|v|^2)A^1_{k}(|E(v)|^2)=0$ and $|E(vf)|^2=E(|v|^2)E(|f|^2)$ for all $f\in \mathcal{D}((EM_v)^k)$. Then for all $f\in \mathcal{D}((EM_v)^k)$ we have
\begin{align*}
\Theta_{T,k}(f)&=\sum_{0\leq i\leq k}(-1)^i
\left(%
\begin{array}{c}
  n \\
  i \\
\end{array}%
\right)\|(EM_v)^i(f)\|^2\\
&=\int_{\Omega}|f|^2d\mu+\sum_{1\leq i\leq k}(-1)^i
\left(%
\begin{array}{c}
  n \\
  i \\
\end{array}%
\right)\int_{\Omega}|E(v)|^{2(i-1)}|E(vf)|^2d\mu\\
&=\int_{\Omega}|f|^2d\mu+\int_{\Omega}\left(\sum_{1\leq i\leq
k}(-1)^i
\left(%
\begin{array}{c}
  n \\
  i \\
\end{array}%
\right)(E(|v|^2))^{2(i-1)}\right)E(|v|^2)E(|f|^2)d\mu\\
&=\int_{\Omega}(1+E(|v|^2)A_{k}(|E(v)|^2))|f|^2d\mu\\
&=0.\\
\end{align*}
This implies that the operator $EM_v$ is $k$-isometry.\\
(iii), (iv). By the same method that is used in (i) and (ii), easily  we get (iii) and (iv).\\
\end{proof}
Here we recall that if the linear
transformation $T=EM_v$ is densely defined on $L^2(\Sigma)$, then $T=EM_v$ is closed and $T^{\ast}=M_{\bar{v}}E$. Also, if $\mathcal{D}(EM_v)$
is dense in $L^2(\Sigma)$ and $v$ is almost every where finite valued, then the operator $EM_v$ is normal if and only if $v\in L^0(\mathcal{A})$ \cite{es}. Hence we have the Remark 3.3 for normal WCT operators.
\begin{rem}\label{r31} Suppose that the operator $EM_v$ is normal and $\mathcal{D}(EM_v)$
is dense in $L^2(\mu)$ for a fixed $k\geq1$. If $|E(f)|^2=E(|f|^2)$ on $S(v)$ for all $f\in \mathcal{D}((EM_v)^k)$,  then:\\

(i) The operator $EM_v$ is $k$-isometry $(k\geq1)$ if and only if $A_{k}(|v|^2)=0$;\\

(ii) The operator $EM_v$ is $k$-expansive if and only if $A_{k}(|v|^2)\leq0$.
\end{rem}
\begin{proof} Since $EM_v$ is normal, then $|E(v)|^2=E(|v|^2)=|v|^2$. Thus by Theorem \ref{t31} we have (i) and (ii).
\end{proof}
Here we give some properties of 2-expansive WCT operators and as a corollary for 2-expansive multiplication operators.
\begin{prop}\label{p31} If
$\mathcal{D}(EM_v)$ is dense in $L^2(\mu)$ and
$EM_v$ is 2-expansive, then:\\

(i) $EM_v$ leaves its domain invariant:\\

(ii) $|E(v)|^{2k}\geq |E(v)|^{2(k-1)}$ a.e. $\mu$ for all $k\geq1$.\\
\end{prop}
\begin{proof} (i). Since $EM_v$ is 2-expansive, we get
that for every $f\in
\mathcal{D}(EM_v)$
\begin{align*}
\|(EM_v)^2(f)\|^2&=\int_{\Omega}|E(v)|^2|E(vf)|^2d\mu\\
&\leq 2\int_{\Omega}|E(vf)|^2d\mu-\int_{\Omega}|f|^2d\mu\\
&<\infty,\\
\end{align*}
so $EM_v(f)\in \mathcal{D}(EM_v)$.\\
(ii) Since $EM_v$ leaves its domain invariant, then
$\mathcal{D}(EM_v)\subseteq \mathcal{D}^{\infty}(EM_v)$.
So by lemma 3.2 (iii) of \cite{jjst} we get that
$\|(EM_v)^k(f)\|^2\geq \|(EM_v)^{k-1}(f)\|^2$ for all
$f\in \mathcal{D}(EM_v)$ and $k\geq1$ we have
$$\int_{\Omega}|E(v)|^{2(k-1)}|E(vf)|^2d\mu\geq \int_{\Omega}|E(v)|^{2(k-2)}|E(vf)|^2d\mu,$$
and so
$$\int_{\Omega}(|E(v)|^{2(k-1)}-|E(v)|^{2(k-2)})|E(vf)|^2d\mu\geq 0,$$
for all $f\in \mathcal{D}(EM_v)$. This leads to $|E(v)|^{2k}\geq |E(v)|^{2(k-1)}$ a.e., $\mu$.
\end{proof}
\begin{cor} If $\mathcal{D}(M_v)$ is
dense in $L^2(\mu)$ and
$M_v$ is 2-expansive, then:\\

(i) $M_v$ leaves its domain invariant:\\

(ii) $v^{2k}\geq v^{2(k-1)}$ a.e. $\mu$ for all $k\geq1$.\\
\end{cor}
Recall that a real-valued map $\phi$ on $\mathbb{N}$ is said to be completely alternating if $\sum_{0\leq i\leq n}(-1)^i\left(%
\begin{array}{c}
  n \\
  i \\
\end{array}%
\right)\phi(m+i)\leq0$ for all $m\geq0$ and $n\geq1$. The next remark  is a direct consequence of Lemma \ref{l31} and Theorem \ref{t31}.
\begin{rem}\label{r32}
 If $\mathcal{D}(EM_v)$
is dense in $L^2(\mu)$ and $k\geq1$ is fixed, then:\\

(i) If the operator $EM_v$ is $k$-hyperexpansive $(k\geq1)$, then $A^0_{n}(|E(v)|^2)\leq0$ for
 $n=1,2,...,k$;\\

(ii) If $(1+E(|v|^2)A^1_{n}(|E(v)|^2))\leq0$  and $|E(vf)|^2=E(|v|^2)E(|f|^2)$ for all $f\in \mathcal{D}(EM_v)^n$ and  $n=1,2,...,k$, then the operator $EM_v$ is $k$-hyperexpansive $(k\geq1)$;\\

(iii) If the operator $EM_v$ is completely hyperexpansive, then\\

(a) the sequence $\{|E(v)(t)|^2\}^{\infty}_{n=0}$ is a completely alternating sequence for almost every $t\in \Omega$,\\

(b) $A^0_{n}(|E(v)|^2)\leq0$ for $n\geq1$.\\

(iv) If $(1+E(|v|^2)A^1_{n}(|E(v)|^2))\leq0$ and $|E(vf)|^2=E(|v|^2)E(|f|^2)$ for all $f\in \mathcal{D}((EM_v)^n)$ and $n\geq1$, then the operator $EM_v$ is completely hyperexpansive.
\end{rem}
%\vspace*{0.3cm} {\bf Proof.} (i) Suppose that the operator $EM_u$ is $k$-hyperexpansive $(k\geq1)$,. So for all $f\in \mathcal{D}((EM_u)^n)$ and $n=1,2,3,..,k$ we have \\
%\begin{align*}
%0&\geq\Theta_{T,n}(f)\\
%&=\int_{X}|f|^2d\mu+\sum_{1\leq i\leq n}(-1)^i
%\left(%
%\begin{array}{c}
%  n \\
%  i \\
%\end{array}%
%\right)\int_{X}|E(u)|^{2(i-1)}|E(uf)|^2d\mu,\\
%\end{align*}
%and so for all $\mathcal{A}$-measurable functions $f\in \mathcal{D}((EM_u)^k)$
%\begin{align*}
%0\geq\int_{X}\left(\sum_{0\leq i\leq n}(-1)^i
%\left(%
%\begin{array}{c}
%  n \\
%  i \\
%\end{array}%
%\right)|E(u)|^{2i}\right)|f|^2d\mu.\\
%\end{align*}
%
%Since $(EM_u)^k$ is densely defined, then we get that $A_{n}(|E(u)|^2)\geq0$.\\
%
%(ii) Let $A_{k}(E(|u|^2))\geq0$. For all $f\in \mathcal{D}((EM_u)^k)$ we have \\
%\begin{align*}
%\Theta_{T,k}(f)&\leq\int_{X}|f|^2d\mu+\int_{X}\left(\sum_{1\leq i\leq
%n}(-1)^i
%\left(%
%\begin{array}{c}
%  n \\
%  i \\
%\end{array}%
%\right)(E(|u|^2))^{i}\right)|f|^2d\mu\\
%&=\int_{X}A_{k}(E(|u|^2))|f|^2d\mu\\
%&\geq0.\\
%\end{align*}
%
%This implies that the operator $EM_u$ is $k$-isometry.\\
%
%(iii), (iv). By the same method that is used in (i) and (ii), easily  we get (iii) and (iv).\\

By  Remark \ref{r32} and some properties of normal WCT operators we get the next remark for $k$-hyperexpansive and completely hyperexpansive normal WCT operators.
\begin{rem}\label{r33}
Let the operator $EM_v$ be normal, $\mathcal{D}(EM_v)$
be dense in $L^2(\mu)$ and $k\geq1$ be fixed. If $|E(f)|^2=E(|f|^2)$ on $S(v)$ for all $f\in \mathcal{D}((EM_v)^k)$, then\\

(i) $EM_v$ is $k$-hyperexpansive $(k\geq1)$ if and only if  $A_{n}(|v|^2)\leq0$ for
$f\in
\mathcal{D}(T^n)$ and $n=1,2,...,k$.\\

(ii)  $EM_v$ is completely hyperexpansive  if and only if  the sequence $\{|u(t)|^2\}^{\infty}_{n=0}$ is a completely alternating sequence for almost every $t\in \Omega$,\\
\end{rem}

%\vspace*{0.3cm} {\bf Proposition 2.3.} If
%$\mathcal{D}((uC_{\phi})^n)$
%is dense in $L^2(\mu)$ for a fixed $n\geq1$, then:\\
%
%(i) $uC_{\phi}$ is $k$-expansive if and only if
%$\triangle_{J,n}(x)\leq 0$ a.e. $\mu$.\\
%
%(ii) $uC_{\phi}$ is $k$-isometry if and only $\triangle_{J,n}(x)=0$ a.e. $\mu$.\\
%
%\vspace*{0.3cm} {\bf Proof.}(i). Since
%$\|(uC_{\phi})^i(f)\|^2=\int_{X}J_i|f|^2d\mu$ for all $f\in
%\mathcal{D}((uC_{\phi})^i)$, we have
%
%$$\sum_{0\leq i\leq n}(-1)^i
%\left(%
%\begin{array}{c}
%  n \\
%  i \\
%\end{array}%
%\right)\|(uC_{\phi})^i(f)\|^2=\sum_{0\leq i\leq n}(-1)^i
%\left(%
%\begin{array}{c}
%  n \\
%  i \\
%\end{array}%
%\right)\int_{X}J_i|f|^2d\mu$$$$=\int_{X}\left(\sum_{0\leq i\leq
%n}(-1)^i
%\left(%
%\begin{array}{c}
%  n \\
%  i \\
%\end{array}%
%\right)J_i\right)|f|^2d\mu=\int_{X}\triangle_{J,n}(x)|f|^2d\mu,
%$$
%for all $f\in \mathcal{D}((uC_{\phi})^n)$. This implies that
%$uC_{\phi}$ is $k$-expansive if and only if
%$\triangle_{J,n}(x)\leq 0$ a.e. $\mu$.\\
%
%(ii) Likewise we have $uC_{\phi}$ is $k$-isometry if and only $\triangle_{J,n}(x)=0$ a.e. $\mu$.\\
%
If all functions $v^{2i}$ for $i=1,. . . ,n$ are finite
valued, then we set
 $$\triangle_{v,n}(x)=\sum_{0\leq i\leq n}(-1)^i
\left(%
\begin{array}{c}
  n \\
  i \\
\end{array}%
\right)|v|^{2i}(t).
$$

Also, if $\mathcal{A}=\Sigma$, then $E=I$. So we have next two corollaries.
\begin{cor} If $\mathcal{D}(M_v)$
is dense in $L^2(\mu)$ for a fixed $n\geq1$, then:\\

(i) $M_v$ is $k$-expansive if and only if
$\triangle_{v,n}(x)\leq 0$ a.e. $\mu$.\\

(ii) $M_v$ is $k$-isometry if and only $\triangle_{v,n}(x)=0$ a.e. $\mu$.
\end{cor}
\begin{cor} Let $\mathcal{D}(M_v)$
be dense in $L^2(\mu)$ and $k\geq1$ be fixed. Then\\

(i) $M_v$ is $k$-hyperexpansive $(k\geq1)$ if and only if  $\triangle_{v,n}(t)\leq0$ a.e., $\mu$ for $n=1,2,...,k$.\\

(ii) $M_v$ is completely hyperexpansive  if and only if  the sequence $\{|u(t)|^2\}^{\infty}_{n=0}$ is a completely alternating sequence for almost every $t\in \Omega$.
\end{cor}
Finally we give some examples.
\begin{exam}
Let $\Omega=[-1,1]$, $d\mu=\frac{1}{2}dx$ and $\mathcal{A}=<\{(-a,a):0\leq a\leq1\}>$ (Sigma algebra generated by symmetric intervals).
Then
 $$E^{\mathcal{A}}(f)(t)=\frac{f(t)+f(-t)}{2}, \ \ t\in \Omega,$$
 where $E^{\mathcal{A}}(f)$ is defined. If $v(t)=e^{t}$, then $E^{\mathcal{A}}(v)(t)=\cosh(t)$ and we have the followings:\\

 1) $E^{\mathcal{A}}M_v$ is densely defined and closed on $L^p(\Omega)$.\\

 2) $\sigma(E^{\mathcal{A}}M_v)=\mathcal{R}(\cosh(t))$.\\

 3) $E^{\mathcal{A}}M_v$ is not 2-expansive, since

 \begin{align*}
 1-2|E(v)|^2(t)+|E(v)|^4(t)&=1-2\cosh^2(t)+\cosh^4(t)\\
 &=(\cosh^2(t)-1)^2\geq0.
 \end{align*}
  \end{exam}
 \begin{exam}
 Let $\Omega=\mathbb{N}$, $\mathcal{G}=2^{\mathbb{N}}$ and let
$\mu(\{t\})=pq^{t-1}$, for each $t\in \Omega=$, $0\leq p\leq 1$ and
$q=1-p$. Elementary calculations show that $\mu$ is a probability
measure on $\mathcal{G}$. Let $\mathcal{A}$ be the
$\sigma$-algebra generated by the partition
$B=\{\Omega_1=\{3n:n\geq1\}, \Omega^{c}_1\}$ of $\Omega$. So, for every $f\in
\mathcal{D}(E^{\mathcal{A}})$ we have
$$E(f)=\alpha_1\chi_{\Omega_1}+\alpha_2\chi_{\Omega^c_1}$$
and direct computations show that
$$\alpha_1(f)=\frac{\sum_{n\geq1}f(3n)pq^{3n-1}}{\sum_{n\geq1}pq^{3n-1}}$$
and
$$\alpha_2(f)=\frac{\sum_{n\geq1}f(n)pq^{n-1}-\sum_{n\geq1}f(3n)pq^{3n-1}}{\sum_{n\geq1}pq^{n-1}-\sum_{n\geq1}pq^{3n-1}}.$$
So, if $u$ and $w$ are real functions on $\Omega$. Then we have the followings:\\

1) If $\alpha_1((|u|^q)^{\frac{p}{q}})\alpha_1(|w|^p)<\infty$ and $\alpha_2((|u|^q)^{\frac{p}{q}})\alpha_2(|w|^p)<\infty$ , then the operator $M_wEM_u$ is  a densely defined and closed operator on $L^p(\Omega)$.\\

2) $\sigma(M_wEM_u)=\{\alpha_1(E(uw)), \alpha_2(E(uw))\}$.
\end{exam}
\begin{exam} Let $\Omega=[0,1]\times
[0,1]$, $d\mu=dtdt'$, $\Sigma$  the  Lebesgue subsets of $\Omega$ and
let $\mathcal{A}=\{A\times [0,1]: A \ \mbox{is a Lebesgue set in}
\ [0,1]\}$. Then, for each $f$ in $L^2(\Sigma)$, $(Ef)(t,
t')=\int_0^1f(t,s)ds$, which is independent of the second
coordinate. Hence for $v(t,t')=t^m$ we get that $v$ is $\mathcal{A}$-measurable and $EM_v$ is $k$-expansive and $k$-isometry if
$$\sum_{0\leq i\leq k}(-1)^i
\left(%
\begin{array}{c}
  k \\
  i \\
\end{array}
\right)x^{2mi}\leq0, \ \ \ \ \sum_{0\leq i\leq k}(-1)^i
\left(%
\begin{array}{c}
  k \\
  i \\
\end{array}
\right)t^{2mi}=0,$$
respectively. This example is valid in the general case as follows:\\
Let $(\Omega_1,\Sigma_1, \mu_1)$ and $(\Omega_2,\Sigma_2, \mu_2)$ be two
$\sigma$-finite measure spaces and $\Omega=\Omega_1\times \Omega_2$,
$\Sigma=\Sigma_1\times \Sigma_2$ and $\mu=\mu_1\times \mu_2$. Put
$\mathcal{A}=\{A\times \Omega_2:A\in \Sigma_1\}$. Then $\mathcal{A}$ is
a sub-$\sigma$-algebra of $\Sigma$. Then for all $f$ in domain
$E^{\mathcal{A}}$ we have
$$E^{\mathcal{A}}(f)(t_1)=E^{\mathcal{A}}(f)(t_1,t_2)=\int_{\Omega_2}f(t_1,s)d\mu_2(s) \ \ \
\mu-a.e.$$ on $\Omega$.\\

Also, if $(\Omega,\Sigma, \mu)$ is a finite measure space and
$k:\Omega\times \Omega\rightarrow \mathbb{C}$ is a $\Sigma\otimes
\Sigma$-measurable function such that
$$\int_{\Omega}|k(.,s)f(s)|d\mu(s)\in L^2(\Sigma)$$
for all $f\in L^2(\Sigma)$. Then the operator
$T:L^2(\Sigma)\rightarrow L^2(\Sigma)$ defined by
$$Tf(t)=\int_{\Omega}k(t,s)f(s)d\mu, \ \ \ \ \ f\in L^2(\Sigma),$$
is called kernel operator on $L^2(\Sigma)$). We show that $T$ is a
weighted conditional type operator.\cite{gd} Since
$L^2(\Sigma)\times \{1\}\cong L^2(\Sigma)$ and
 $vf$ is a $\Sigma\otimes \Sigma$-measurable function, when $f\in
L^2(\Sigma)$. Then by taking $v:=k$ and $f'(t,s)=f(s)$, we get
that
\begin{align*}
E^{\mathcal{A}}(vf)(t)&=E^{\mathcal{A}}(vf')(t,s)\\
&=\int_{\Omega}v(t,t')f'(,t')d\mu(t')\\
&=\int_{\Omega}v(t,t')f(t')d\mu(t')\\
&=Tf(t).\\
\end{align*}
Hence $T=EM_v$, i.e, $T$ is a weighted conditional type operator. This means all assertions of this paper are valid for a class of integral type operators.
\end{exam}

%(a) Let us consider $\Sigma$, the sigma algebra of Lebesgue measurable sets,  Lebesgue measure $\mu$ on $[0,1)$ and let $T(x)=(x+2^{-n})mod(1)$, $n\in \mathbb{N}$, $0\leq x<1$. Let
%
%$$\mathcal{B}:=\{B\in \Sigma: T(B)=B\}.$$
%Then a $\Sigma$-measurable function $f$ is measurable with respect to $\mathcal{B}$ if and only if the equality $f=f\circ T$ holds. In adition, one has\\
%$$E^{\mathcal{B}}f=$$


\begin{thebibliography}{99}
%\bibitem{jjab} J. Jablonski, Hyperexpansive composition
%operators, Math. Proc. Camb. Phil. Soc. {\bf 135} (2003), 513-526.
%\bibitem{al} A. Aleman, The multiplication operators on Hilbert spaces of analtyic functions, Habilitationsschrift,
%FernUniversitat Hagen (1993).
%
%\bibitem{aw} A. Aluthge and D. Wang, $w$-hyponormal operators,
%Integr. Equat. Oper. Th {\bf 36} (2000), 1-10.
%
%\bibitem{at} A. Athavale, On completely hyperexpansive operators, Proc. Amer. Math. Soc. {\bf 124} (1996),
%3745–3752.
%
%
%
%\bibitem{Bill} P. Billingsley, Probability and measure, John Wiley
%and Sons, New York. 1995.

\bibitem{dhd} P.G. Dodds, C.B. Huijsmans and B. De Pagter,
characterizations of conditional expectation-type operators,
Pacific J. Math. {\bf 141}(1) (1990), 55-77.

\bibitem{dou}
 R. G. Douglas, Contractive projections on an $L\sb{1}$ space,
 Pacific J. Math. {\bf 15} (1965), 443-462.

\bibitem{es} Y. Estaremi, Unbounded weighted conditional expectation operators, Complex Anal. Oper. Theory,to appear 2015.

\bibitem{ej} Y. Estaremi and M.R. Jabbarzadeh, Weighted lambert type operators on
$L^{p}$-spaces, Oper. Matrices {\bf 1} (2013), 101-116.


\bibitem{gd}
J. J. Grobler and B. de Pagter, Operators representable as
multiplication-conditional expectation operators, J. Operator
Theory {\bf 48} (2002), 15-40.


%\bibitem{her}
%J. Herron, Weighted conditional expectation operators, Oper. Matrices {\bf 1} (2011), 107-118.

\bibitem {jab} Z. J. Jablonski. Complete hyperexpansivity, subnormality and inverted boundedness conditions;
Integral Equations Operator Theory {\bf 44} (2002) 316–336.

%\bibitem{jjab} J. Jablonski, Hyperexpansive composition
%operators, Math. Proc. Camb. Phil. Soc. {\bf 135} (2003), 513-526.

\bibitem{jjst} Z. J. Jablonski and J. Stochel, Unbounded
2-hyperexpansive operators, Proc. Edinburgh Math. Soc. {\bf 44}
(2001), 613-629.

\bibitem{lam}
A. Lambert, $L^p$ multipliers and nested sigma-algebras, Oper.
Theory Adv. Appl. {\bf 104} (1998),  147-153.


\bibitem{mo}
Shu-Teh Chen, Moy,  Characterizations of conditional expectation
as a transformation on function spaces, Pacific J. Math. {\bf 4}
(1954), 47-63.

%\bibitem{rao1}
%M. M. Rao, Conditional measures and operators, J. Multivate Anal.
%{\bf 5} (1975), 330-413.
%
%\bibitem{rao2}
%M. M. Rao, Two characterizations of conditional probability, Proc.
%Amer. Math. Soc. {\bf 19} (1976), 75-80.

\bibitem{rao}
M. M. Rao, Conditional measure and applications, Marcel Dekker,
New York, 1993.

%\bibitem{rich} S. Richter. Invariant subspaces of the Dirchlet shift. J. Reine Angew. Math. {\bf 386} (1988),
%205-220.
%
%\bibitem {shi} S. Shimorin, Wold-type decompositions and wandering subspaces for operators close to isometries.
%J. Reine Angew. Math. {\bf 51} (2001), 147-189.


\bibitem{sho} V. M. Sholapurkar and A. Athavale, Completely and alternatingly hyperexpansive operators.
J. Operator Theory {\bf 43} (2000), 43-68.


%\bibitem{ty}
%T. Yamazaki, An expression of spectral radius via Aluthge
%transformation, Pro. A. M. S. {\bf 4} (2001), 1131-1137.

\bibitem{z}
A. C. Zaanen, Integration, 2nd ed., North-Holland, Amsterdam,
1967.

\end{thebibliography}
\end{document}